\documentclass[smallextended]{svjour3-ppn}
\smartqed

\usepackage{amssymb,graphicx,amsmath,setspace,graphicx}
\usepackage[colorlinks=true]{hyperref}
\hypersetup{urlcolor=blue, citecolor=red, linkcolor=blue}

\renewcommand{\S}{{\mathbb S}^{d-1}}
\newcommand{\nrm}[2]{\|{#1}\|_{\L^{#2}(\R^d)}}

\newcommand{\nrmcnd}[2]{\|{#1}\|_{\L^{#2}(\mathcal C)}}
\newcommand{\nrmC}[2]{\|{#1}\|_{#2}}
\newcommand{\icnd}[1]{\int_{\mathcal C}{#1}\;dy}
\newcommand{\iC}[1]{\int_{\mathcal C}{#1}\;d\mu}
\newcommand{\be}[1]{\begin{equation}\label{#1}}
\newcommand{\ee}{\end{equation}}
\renewcommand{\S}{{\mathbb S^{d-1}}}
\newcommand{\C}[1]{\mathsf C_{\rm #1}}
\newcommand{\K}[1]{\mathsf K_{\rm #1}}

\newcommand{\N}{\mathbb N}
\newcommand{\R}{{\mathbb R}}
\renewcommand{\L}{\mathrm L}
\renewcommand{\H}{\mathrm H}
\renewcommand{\(}{\left(}
\renewcommand{\)}{\right)}

\renewcommand{\S}{{\mathbb S^{d-1}}}

\newcommand{\ird}[1]{\int_{\R^d}{#1}\;dx}

\newcommand{\nrmrd}[2]{\|{#1}\|_{\L^{#2}(\R^d)}}

\newcommand{\xa}{\alpha}

\newcommand{\xg}{\gamma}

\newcommand{\xl}\lambda
\newcommand{\xL}\Lambda

\newcommand{\bea}{\begin{eqnarray}}
\newcommand{\eea}{\end{eqnarray}}
\newcommand{\bean}{\begin{eqnarray*}}
\newcommand{\eean}{\end{eqnarray*}}
\newcommand{\la}{\label}

\usepackage{color}

\newcommand{\subjclass}[1]{\thanks{2010 \emph{Mathematics Subject Classification.} #1}}

\begin{document}

\title{Rigidity results with applications to best constants and symmetry of Caffarelli-Kohn-Nirenberg and logarithmic Hardy inequalities}
\titlerunning{Symmetry in Caffarelli-Kohn-Nirenberg and logarithmic Hardy inequalities}

\author{Jean Dolbeault \and Maria J. Esteban \and Stathis Filippas \and Achilles Tertikas}

\institute{
J. Dolbeault \at Ceremade, Univ. Paris-Dauphine, Pl. de Lattre de Tassigny, 75775 Paris C\'edex~16, France
\email{dolbeaul@ceremade.dauphine.fr}
\and
M.J. Esteban \at Ceremade, Univ. Paris-Dauphine, Pl. de Lattre de Tassigny, 75775 Paris C\'edex~16, France
\email{esteban@ceremade.dauphine.fr}
\and
S. Filippas \at Department of Mathematics, Univ. of Crete, Knossos Avenue, 714 09 Heraklion
\& Institute of Applied and Computational Mathematics, FORTH, 71110 Heraklion, Crete, Greece
\email{filippas@tem.uoc.gr}
\and
A. Tertikas \at Department of Mathematics, Univ. of Crete, Knossos Avenue, 714 09 Heraklion
\& Institute of Applied and Computational Mathematics, FORTH, 71110 Heraklion, Crete, Greece
\email{tertikas@math.uoc.gr}
}

\date{\today}

\maketitle

\begin{abstract}
We take advantage of a rigidity result for the equation satisfied by an extremal function associated with a special case of the Caffarelli-Kohn-Nirenberg inequalities to get a symmetry result for a larger set of inequalities. The main ingredient is a reparametrization of the solutions to the Euler-Lagrange equations and estimates based on the rigidity result. The symmetry results cover a range of parameters which go well beyond the one that can be achieved by symmetrization methods or comparison techniques so far.
\end{abstract}


\noindent\subjclass{26D10 \and 46E35 \and 35J20 \and 49J40}

\maketitle
\thispagestyle{empty}

\section{Introduction and main results}

Let $2^*:=\infty$ if $d=1$, $2$, and $2^*:=2\,d/(d-2)$ if $d\ge3$. Define 
\[
\vartheta(p,d):=\frac{d\,(p-2)}{2\,p}, \hspace{1,5cm} a_c:=\frac{d-2}{2} \ ,
\]
and consider the space $\mathcal D_a^{1,2}(\R^d)$ obtained by completion of $\mathcal D(\R^d\setminus\{0\})$ with respect to the norm $v\mapsto\nrm{\,|x|^{-a}\,\nabla v\,}2^2$. We will be concerned with the following two families of inequalities

\vspace{3mm}

\noindent{\bf Caffarelli-Kohn-Nirenberg Inequalities (CKN) } {\rm \cite{Caffarelli-Kohn-Nirenberg-84}}
Let $d\ge 1$. For any $p\in[2, 2^*]$ if $d\ge3$ or $p\in[2, 2^*)$ if $d=1$, $2$, for any $\theta\in[\vartheta(p,d),1]$ with $\theta>1/2$ if $d=1$, there exists a positive constant $\C{CKN}(\theta,p,a)$ such that 
\be{Ineq:GenInterp}
\(\;\ird{\frac{|v|^p}{|x|^{b\,p}}}\)^\frac2p\le\C{CKN}(\theta,p,a)\(\;\ird{\frac{|\nabla v|^2}{|x|^{2\,a}}}\)^{\!\theta}\(\;\ird{\frac{|v|^2}{|x|^{2\,(a+1)}}}\)^{\!1-\theta}
\ee
holds true for any $v\in\mathcal D^{1,2}_{a}(\R^d)$. Here $a$, $b$ and $p$ are related by $b=a-a_c+d/p$, with the restrictions $a\le b\le a+1$ if $d\ge3$, $a<b\le a+1$ if $d=2$ and $a+1/2<b\le a+1$ if $d=1$. Moreover, the constants $\C{CKN}(\theta,p,a)$ are uniformly bounded outside a neighborhood of $a=a_c$.

In \cite{DDFT}, a new class of inequalities, called {\it weighted logarithmic Hardy inequalities}, was considered. These inequalities can be obtained from \eqref{Ineq:GenInterp} by taking $\theta=\gamma\,(p-2)$ and passing to the limit as $p\to2_+$.

\vspace{3mm}\noindent{\bf Weighted Logarithmic Hardy Inequalities (WLH)}{\rm \cite{DDFT}}
Let $d\ge 1$, $a<a_c$, $\gamma\ge d/4$ and $\gamma>1/2$ if $d=2$. Then there exists a positive constant $\C{WLH}(\gamma,a)$ such that, for any $v\in\mathcal D^{1,2}_{a}(\R^d)$ normalized by 
\[
\ird{|x|^{-2\,(a+1)}\,|v|^2}=1 \ ,
\] we have
\be{Ineq:GLogHardy}
\ird{\frac{|v|^2}{|x|^{2\,(a+1)}}\,\log \(|x|^{2\,(a_c-a)}\,|v|^2\)}\le2\,\gamma\,\log\left[\C{WLH}(\gamma,a)\ird{\frac{|\nabla v|^2}{|x|^{2\,a}}}\right].
\ee
Moreover, the constants $\C{WLH}(\gamma,a)$ are uniformly bounded outside a neighborhood of $a=a_c$.

It is very convenient to reformulate the Caffarelli-Kohn-Nirenberg inequality in cylindrical variables as in~\cite{Catrina-Wang-01}. By means of the Emden-Fowler transformation
\[
s=\log|x|\in\R\ ,\quad\omega=\frac{x}{|x|}\in\S\ ,\quad y=(s,\omega)\ ,\quad u(y)=|x|^{a_c-a}\,v(x)\ ,
\]
Inequality~\eqref{Ineq:GenInterp} for $v$ is equivalent to a Gagliardo-Nirenberg-Sobolev inequality for the function $u$ on the cylinder $\mathcal C:=\R\times\S$:
\be{Ineq:Gen_interp_Cylinder}
\K{CKN}(\theta, p, \Lambda)\,\nrmcnd up^2\le \(\nrmcnd{\nabla u}2^2\!+\!\Lambda\,\nrmcnd u2^2\)^\theta \nrmcnd u2^{2\,(1-\theta)}\quad\kern-1.66pt\forall\,u\in\H^1(\mathcal C)\ .
\ee
Here and throughout the rest of the work we set
\[
\Lambda:=(a_c-a)^2 \ .
\]
Similarly, with $u(y)=|x|^{a_c-a}\,v(x)$, Inequality \eqref{Ineq:GLogHardy} is equivalent to
\be{Ineq:GLogHardy-w}
\icnd{|u|^2\,\log |u|^2}\le2\,\gamma\,\log\left[\frac{1}{\K{WLH}(\gamma,\xL)}\Big(\nrmcnd{\nabla u}2^2+\Lambda\Big)\right] \ ,
\ee
for any $u\in\H^1(\mathcal C)$ such that $\nrmcnd u2=1$. In both cases, we consider on~$\mathcal C$ the measure
$d\mu=|\S|^{-1}d\omega\,ds$ obtained by normalizing the surface of $\S$ to $1$ (that is, the uniform probability measure), tensorized with the usual Lebesgue measure on the axis of the cylinder. 

We are interested in \emph{symmetry} and \emph{symmetry breaking} issues: when do we know that equality in \eqref{Ineq:GenInterp} and~\eqref{Ineq:GLogHardy} is achieved by radial functions or, alternatively, by functions depending only on $s$ in~\eqref{Ineq:Gen_interp_Cylinder} and~\eqref{Ineq:GLogHardy-w}? Related with inequality~(\ref{Ineq:Gen_interp_Cylinder}) is the Rayleigh quotient: 
\[
\mathcal Q_\Lambda^\theta[u]:=\frac{\(\nrmC{\nabla u}2^2+\Lambda\,\nrmC u2^2\)^\theta\,\nrmC u2^{2\,(1-\theta)}}{\nrmC up^2}\ .
\]
Here $\nrmC uq:=\(\iC{|u|^q}\)^{1/q}$. Then \eqref{Ineq:Gen_interp_Cylinder} and~\eqref{Ineq:GLogHardy-w} are equivalent to state that
\begin{eqnarray*}
&&\K{CKN}(\theta,p,\Lambda)=\inf_{u\in\H^1(\mathcal C)\setminus\{0\}}\mathcal Q_\Lambda^\theta[u]\ ,\\
&&\K{WLH}(\gamma,\Lambda)=\inf_{\begin{array}{c}\scriptstyle u\in\H^1(\mathcal C)\setminus\{0\}\cr\scriptstyle\nrmC u2=1\end{array}}\(\nrmC{\nabla u}2^2+\Lambda\)\,e^{-\frac 1{2\,\gamma}\iC{|u|^2\,\log |u|^2}}\ .
\end{eqnarray*}
Let $\K{CKN}^*(\theta,p,\Lambda)$
and $\K{WLH}^*(\gamma,\Lambda)$ be the corresponding values of the infimum when the set of minimization is restricted to functions depending only on $s$. The main interest of introducing the measure $d\mu$ is that $\K{CKN}^*(\theta,p,\Lambda)$ and $\K{WLH}^*(\gamma,\Lambda)$ are independent of the dimension and can be computed for $d=1$ by solving the problem on the real line $\R$. 

Radial symmetry of $v=v(x)$ means that $u=u(s,\omega)$ is independent of~$\omega$. Up to translations in $s$ and a multiplication by a constant, the optimal functions in the class of functions depending only on~$s\in\R$ solve the equation
\[\label{eq:onedim}
-u_*''+\Lambda\,u_*=u_*^{p-1}\quad\mbox{in}\quad \R
\]
if $\theta=1$. See Section~\ref{Sec:Parametrization} if $\theta<1$. Up to translations in $s$, non-negative solutions of this equation are all equal to the function
\be{ustar}
u_*(s):=\frac A{\big[\cosh(B\,s)\big]^\frac 2{p-2}}\quad\forall\,s\in\R\ ,
\ee
with $A^{p-2}=\tfrac p2\,\Lambda$ and $B=\tfrac12\,\sqrt\Lambda\,(p-2)$. The uniqueness up to translations is a standard result (see for instance \cite[Proposition B.2]{Dolbeault06082014} for a proof).

The symmetry breaking issue is now reduced to the question of knowing whether the inequalities
\be{K-Kstar}
\K{CKN}(\theta,p,\Lambda)\le\K{CKN}^*(\theta,p,\Lambda)\quad\mbox{and}\quad\K{WLH}(\gamma,\Lambda)\le\K{WLH}^*(\gamma,\Lambda)
\ee
are strict or not, when $d\ge2$. Symmetry breaking occurs if the inequality is strict and then optimal functions \emph{are not} symmetric (symmetric means: depending only on $s$ in the setting of the cylinder, or on $|x|$ in the case of the Euclidean space). In \cite[pp.~2048 and 2057]{DDFT}, the values of the symmetric constants have been computed. They are given by
\be{EqnK}
\textstyle \K{CKN}^*(\theta,p, \xL):=\left[\frac{2p \theta +2-p}{(p-2)^2}\right]^\frac{p-2}{2\,p}\left[\frac{2 p \theta}{2 p \theta +2 -p}\right]^\theta\left[\frac {p+2}4\right]^\frac{6-p}{2\,p}
\left[\frac{\sqrt\pi\;\Gamma\(\frac2{p-2}\)}{\Gamma\(\frac2{p-2}+\frac 12\)}\right]^\frac{p-2}p \kern-5pt\Lambda^{\theta-\frac{p-2}{2\,p}}
\ee
and
\begin{eqnarray*}
&&\textstyle\K{WLH}^*(\gamma,\Lambda)=\frac{\gamma\,\(8\pi^{d+1}\,e\)^\frac1{4\gamma}}{\Gamma\(\frac{d}2\)^\frac1{2\gamma}}\,
\(\frac{4 \xL}{4\gamma-1}\)^\frac{4\gamma-1}{4\gamma}\;\mbox{if}\;\gamma>\frac14\ ,\\
&&\textstyle\K{WLH}^*(\gamma,\Lambda)=\frac{2\pi^{d+1}\,e}{\Gamma\(\frac{d}2\)^2} \;\;\;\;\;~~~~~~~~~~\mbox{if}\;\;~\;\gamma=\frac14\ .
\end{eqnarray*}

Let
\be{lambdastar}\textstyle
\Lambda_{\rm FS}(\theta,p,d):=4\,\frac{d-1}{p^2-4}\,\frac{(2\,\theta-1)\,p+2}{p+2}\quad\mbox{and}\quad\Lambda_\star(1,p,d):=\frac14\,(d-1)\,\frac{6-p}{p-2}\ .
\ee
We will define $\Lambda_\star(\theta,p,d)$ for $\theta<1$ later in the Introduction. Symmetry breaking occurs for any $\Lambda>\Lambda_{\rm FS}$ according to a result of V.~Felli and M.~Schneider in \cite{Felli-Schneider-03} for $\theta=1$ and in \cite{DDFT} for $\theta<1$ (also see \cite{Catrina-Wang-01} for previous results and \cite{MR2437030} if $d=2$ and $\theta=1$). This symmetry breaking is a straightforward consequence of the fact that for $\Lambda>\Lambda_{\rm FS}$, the symmetric \emph{optimals} are saddle points of an energy functional, and thus cannot be even local minima. As a consequence, we know that $\K{CKN}(\theta,p,\Lambda)<\K{CKN}^*(\theta,p,\Lambda)$ if $\Lambda>\Lambda_{\rm FS}(\theta,p,d)$.

Concerning the log Hardy inequality, it was shown in \cite{DDFT} that symmetry breaking occurs, that is, 
$\K{WLH}(\gamma,\Lambda) < \K{WLH}^*(\gamma,\Lambda)$,  when either $d=2$ and $\xg>1/2$ or 
$d\geq 3$ and $\xg \geq d/4$ provided that 
\[
\xL > (d-1)\( \xg- \tfrac14 \) \ .
\]

Concerning symmetry, if $\theta=1$, from \cite{DEL2011}, we know that symmetry holds for CKN for any $\Lambda\le\Lambda_\star(1,p,d)$. The precise statement goes as follows.
\begin{theorem}\label{Thm:DEL2011}{\rm \cite{DEL2011}} Let $d\ge 2$. For any $p\in[2, 2^*]$ if $d\ge3$ or $p\in[2,\infty)$ if $d=2$, under the conditions
\[
0<\mu\le\Lambda_\star(1,p,d) \hspace{8mm} {\rm and} \hspace{8mm} \mathcal Q_\mu^1[u]\le\K{CKN}^*(1,p,\mu) \ ,
\]
the solution of
\be{EL}
-\,\Delta u+\mu\,u=u^{p-1}\quad\mbox{on}\quad\mathcal C
\ee
is given by the one-dimensional equation, written on $\R$. It is unique, up to translations.\end{theorem}
Theorem~\ref{Thm:DEL2011} is a \emph{rigidity result}. In \cite{DEL2011}, the proof is given for a minimizer of~$\mathcal Q_\mu^1$, which therefore satisfies $\mathcal Q_\mu^1[u]\le\K{CKN}^*(1,p,\mu)$, but the reader is invited to check that only the latter condition is used in the proof. The proof is based on a chain of estimates which involve optimal interpolation inequalities on the sphere and the Keller-Lieb-Thirring inequality. These inequalities turn out to be equalities, and equality in each of the inequalities is shown to imply that the solution only depends on $s$ (no angular dependence). The result of Theorem~\ref{Thm:DEL2011} gives a sufficient condition for symmetry when $\theta=1$. We shall say that \emph{any minimizer is symmetric} if it is given by~\eqref{ustar}, up to multiplications by constants and translations. 
\begin{theorem}\label{Cor:DEL2011}{\rm \cite{DEL2011}} Let $d\ge 2$. For any $p\in[2,2^*]$ if $d\ge3$ or any $p\in[2,\infty)$ if $d=2$, if
$0<\Lambda\le\Lambda_\star(1,p,d)$, then $\K{CKN}(1,p,\Lambda)=\K{CKN}^*(1,p,\Lambda)$ and any minimizer is symmetric.\end{theorem}

In~\cite{DEL2011}, the case $\theta<1$ is also considered. According to~\cite[Theorem~9]{DEL2011}, for any $d\ge 3\,$, any $p\in(2,2^*)$ and any $\theta\in[\vartheta(p,d),1)\,$, we have the estimate
\be{K:Interval}
\mathfrak C(\theta,p)^{-\frac{2\,\theta}{q+2}}\,\mathsf K_{\rm CKN}^*(\theta,\Lambda,p)\le\mathsf K_{\rm CKN}(\theta,\Lambda,p)\le\mathsf K_{\rm CKN}^*(\theta,\Lambda,p)
\ee
where $q:=\frac{2\,(p-2)}{(2\,\theta-1)\,p+2}$ and
\[
\mathfrak C(\theta,p):=\tfrac{(p+2)^\frac{p+2}{(2\,\theta-1)\,p+2}}{(2\,\theta-1)\,p+2}\,\(2-\tfrac p2\,(1-\theta)\)^{1-\frac q2}\\
\cdot\(\frac{\Gamma(\frac p{p-2})}{\Gamma(\frac{\theta\,p}{p-2})}\)^{\!2\,q}\,\(\frac{\Gamma(\frac{2\,\theta\,p}{p-2})} {\Gamma(\frac{2\,p}{p-2})}\)^{\!q}
\]
under the condition $a_c^2<\Lambda\le\frac{(d-1)}{\mathfrak C(\theta,p)}\,\frac{(2\,\theta-3)\,p+6}{4\,(p-2)}$. If $\theta=1$, the equality case in the last inequality characterizes $\Lambda_\star(1,p,d)$ as defined in \eqref{lambdastar}. However \eqref{K:Interval} does not give a range for symmetry unless $\theta=1$.

Much more is known. According to \cite{DELT09,springerlink:10.1007/s00526-011-0394-y}, there is a continuous curve $p\mapsto\Lambda_{\rm s}(\theta,p,d)$ with $\lim_{p\to2_+}\Lambda_{\rm s}(\theta,p,d)=\infty$ and $\Lambda_{\rm s}(\theta,p,d)>a_c^2$ for any \hbox{$p\in(2,2^*)$} such that symmetry holds for any $\Lambda\le\Lambda_{\rm s}(1,p,d)$ and there is symmetry breaking if \hbox{$\Lambda>\Lambda_{\rm s}(1,p,d)$}, for any $\theta\in[\vartheta(p,d),1)$. Additionally, we have that $\lim_{p\to2^*}\Lambda_{\rm s}(1,p,d)=a_c^2$ if $d\ge3$ and, if $d=2$, $\lim_{p\to\infty}\Lambda_{\rm s}(1,p,d)=0$ and $\lim_{p\to\infty}p^2\Lambda_{\rm s}(1,p,d)=4$. The existence of this function~$\Lambda_{\rm s}$ has been proven in an indirect way, and it is not explicitly known. It has been a long-standing question to decide whether the curves $p\to\Lambda_{\rm s}(\theta,p,d)$ and the curve $p\to\Lambda_{\rm FS}(\theta,p,d)$ coincide or not. This is still an open question, at least for $\theta=1$. For $\theta<1$, and for some specific values of $p$, it has been shown that, in some cases, $\Lambda_{\rm s}(\theta,p,d)<\Lambda_{\rm FS}(\theta,p,d)$; see~\cite{springerlink:10.1007/s00526-011-0394-y} for more details, as well as some symmetry results based on symmetrization techniques. A scenario based on numerical computations and asymptotic expansions at the point where non-symmetric positive solutions bifurcate from the symmetric ones has been proposed; see \cite{Oslo,Freefem,DE2012} for details.

\medskip Our interest in this work is to establish symmetry of the minimizers of CKN for $\theta <1$ as well as of the log Hardy inequalities, thus identifying the corresponding sharp constants.

Our first result is an extension of Theorem~\ref{Cor:DEL2011} to the case $\theta<1$. Our goal is to give explicit estimates of the range for which symmetry holds. This requires some notations and a preliminary result. We set
\be{m5}
\Pi^*(\theta,p,q):=\(\frac{\K{CKN}^*(\theta,p,1)}{\K{CKN}^*(1,q,1)^{ \frac{q\,(p-2)}{p\,(q-2)} }}\)^{\frac1{\theta-\frac{q\,(p-2)}{p\,(q-2)}}} \ .
\ee
Next we define
\be{qstar}
q^*=q^*(\theta,p):=\frac{2\,p\,\theta}{2-p\,(1-\theta)}\ .
\ee
The condition $\theta> \frac{q\,(p-2)}{p\,(q-2)} $ is equivalent to $q>q^*(\theta,p)$ and we can notice that $p<q^*(\theta,p)<2^*$ for any $\theta \in (\vartheta(p,d), 1)$.
For $d \geq 3$ we define
\[\label{nonexplicit}
\Lambda_1(\theta,p,d):=\max_{q\in(q^*,2^*)}\;\min\left\{\Lambda_\star(1,q,d),\frac{\theta\,\Lambda_\star(1,p,d)}{(1-\theta)\,\Pi^*(\theta,p,q)+\theta}\right\}\ ,
\]
whereas for $d=2$
\[
\Lambda_1(\theta,p,2):=\max_{q\in(q^*,6)}\;\min\left\{\Lambda_\star(1,q,2),\frac{\theta\,\Lambda_\star(1,p,2)}{(1-\theta)\,\Pi^*(\theta,p,q)+\theta}\right\}\ .
\]
Next, we can also define
\be{Ntp}
\mathsf N(\theta,p):=\frac{\big(\K{CKN}^*(\theta,p,1)\big)^{1/\theta}}{\K{CKN}^*\big(1,q^*(\theta,p),1\big)}\ .
\ee
We refer to Section~\ref{Sec:Lambda2} for an explicit expression of $\mathsf N(\theta,p)$. We introduce the exponent
\be{beta}
\beta=\beta(\theta,p):=1-\frac{p-2}{2\,p\,\theta}\ .
\ee
For $2<p<6$ and $\theta\in(\vartheta(p,3),1)$ we denote by $\mathsf x^*=\mathsf x^*(\theta,p)$ the unique root
of the equation
\[
\theta\,(6-p)\big(x^\beta-\mathsf N\big)\,x-\big(2\,p\,\theta-3\,(p-2)\big)\(\theta\,\big(x^\beta-\mathsf N\big)+(1-\theta)\,(x-1)\,\mathsf N \) =0 \ ,
\]
in the interval $(\mathsf N^{1/\beta},\infty)$ for $\mathsf N=\mathsf N(\theta, p)$, see Lemma~\ref{Lem:xstar}
in Section~\ref{Sec:Lambda2}. Next we define 
\[\label{Lambda2-beta}
\Lambda_2(\theta,p,d):=\frac{\Lambda_\star(1,q^*,d)}{ \mathsf x^*(\theta,p)} = \frac14\,(d-1)\,\frac{2\,p\,\theta-3\,(p-2)}{(p-2)\,\mathsf x^*(\theta,p)}\ ,
\]
and
\[
\Lambda_\star(\theta,p,d):=\max\Big\{\Lambda_1(\theta,p,d),\Lambda_2(\theta,p,d)\Big\}\ .
\]
\begin{theorem}\label{Thm:Main1} Suppose that either $d=2$ and $p\in(2,6)$ or else $d \geq 3$ and $p \in(2,2^*)$. Then 
\[
\K{CKN}(\theta ,p,\Lambda)=\K{CKN}^*(\theta ,p,\Lambda) \ ,
\]
and any minimizer of CKN \eqref{Ineq:Gen_interp_Cylinder} is \emph{symmetric} provided that one
of the following conditions is satisfied:
\begin{enumerate}
\item[(i)] $d=2$, $\theta\in(\vartheta(p,2),1)$ and $0<\Lambda\le\Lambda_1(\theta,p,2)$.
\item[(ii)] $d=2$, $\theta\in(\vartheta(p,3),1)$ and $0<\Lambda\le\Lambda_\star(\theta,p,2)$,
\item[(iii)] $d\ge3$, $\theta=\vartheta(p,d)$ and $0<\Lambda\le\Lambda_2(\theta,p,d)$,
\item[(iv)] $d\ge3$, $\theta\in(\vartheta(p,d),1)$ and $0<\Lambda\le\Lambda_\star(\theta,p,d)$ \ .
\end{enumerate}
\end{theorem}
Our definition of $\Lambda_\star(\theta,p,d)$ for $\theta<1$ is consistent with the definition of $\Lambda_\star(1,p,d)$ given in \eqref{lambdastar} because
\[
\lim_{\theta\to1}\Lambda_1(\theta,p,d)=\lim_{\theta\to1}\Lambda_2(\theta,p,d)=\Lambda_\star(1,p,d)\ .
\]
One of the drawbacks in the definition of $\Lambda_2(\theta,p,d)$ is that $\mathsf x^*(\theta,p)$ given by Lemma~\ref{Lem:xstar} is not explicit. For an explicit estimate of $\Lambda_2(\theta,p,d)$ see Proposition~\ref{Prop:Estim1} in Section~\ref{Sec:Numerics}.

\medskip By passing to the limit as $p\to2_+$ in the criterion $\Lambda\le\Lambda_2(\theta,p,d)$, we also obtain an explicit condition for symmetry in the weighted logarithmic Hardy inequalities. For any $\mathsf N_0>1$, consider the smallest root $x>\mathsf N_0^{1/\beta_0}$ of
\[
4\,\gamma\,x^{\beta_0+1}-(8\,\gamma-3)\,\mathsf N_0 \,x+(4\,\gamma-3)\,\mathsf N_0 \,=0\quad\mbox{with}\quad\beta_0=1-\frac1{4\,\gamma}
\]
and denote it by $\mathsf x_0^*(\gamma)$ if $\mathsf N_0=\mathsf N_0(\gamma):=\lim_{p\to2_{+}}\mathsf N(\gamma\,(p-2),p)$. An elementary but tedious computation shows that
\be{N0}
\mathsf N_0(\gamma)=2^{1-\frac3{4\,\gamma}}\,e^\frac1{4\,\gamma}\,\frac{(2\,\gamma-1)^{1-\frac1\gamma}}{(4\,\gamma-1)^{1-\frac3{4\,\gamma}}}\(\frac{\Gamma\(2\,\gamma-\frac12\)}{\Gamma\(2\,\gamma-1\)}\)^{\!\frac 1{2\,\gamma}}.
\ee
Let us define
\be{Lambda0}
\Lambda_0(\gamma,d):=\frac{(d-1)\,\(\gamma-3/4\)}{\mathsf x_0^*(\gamma)}\ .
\ee
We then have
\begin{theorem}\label{Cor:WLH} Assume that either $d=2$ or $3$ and $\gamma>3/4$, or $d\ge4$ and $\gamma\ge d/4$. Then 
\[
\textstyle\K{WLH}(\gamma,\Lambda)=\textstyle\K{WLH}^*(\gamma,\Lambda) \ ,
\] and any minimizer of~\eqref{Ineq:GLogHardy-w} is \emph{symmetric} provided that
\[
0< \Lambda\le\Lambda_0(\gamma,d) \ .
\]\end{theorem}
For an explicit estimate of $\Lambda_0(\xg, d)$ see Proposition~\ref{Prop:Estim2} in Section~\ref{Sec:Numerics}.

Theorem~\ref{Thm:Main1} provides us with a \emph{rigidity result}, which is stronger than a simple symmetry result. As a consequence, our estimates of Theorem~\ref{Thm:Main1} for the symmetry region cannot be optimal.
\begin{theorem}\label{Thm:Main2} Suppose that either $d = 2$ and $p\in(2,6)$ or else $d \ge 3$ and $p\in(2,2^*)$. If $\theta>\vartheta(p,\min\{3,d\})$, then 
\[
\Lambda_\star(\theta,p,d)<\Lambda_{\rm s}(\theta,p,d)\le\Lambda_{\rm FS}(\theta,p,d)\ .
\]
If either $d=3$ and $\theta=\vartheta(p,3)$, or $d=2$ and $\theta>0$, then
\[
\Lambda_2(\theta,p,d)<\Lambda_{\rm s}(\theta,p,d)\le\Lambda_{\rm FS}(\theta,p,d)\ .
\]\end{theorem}
It can be conjectured that $\Lambda_{\rm s}(\theta,p,d)=\Lambda_{\rm FS}(\theta,p,d)$ holds in the limit case $\theta=1$, and probably also for $\theta$ close enough to $1$, on the basis of the numerical results of \cite{Freefem} and the formal computations of \cite{DE2012}. On the other hand, it is known from \cite{springerlink:10.1007/s00526-011-0394-y} that $\Lambda_{\rm s}(\theta,p,d)<\Lambda_{\rm FS}(\theta,p,d)$ when $\theta-\vartheta(p,d)$ is small enough, at least for some values of $p$ and $d$.

The expressions involved in the statement of Theorem~\ref{Thm:Main1} look quite technical, but they are interesting for two reasons:
\begin{itemize}
\item[$\bullet$] Theorem~\ref{Thm:Main1} determines a range for symmetry which goes well beyond what can be achieved using standard methods and is somewhat unexpected in view of the estimate of \cite[Theorem~9]{DEL2011}. It is a striking observation that the reparametrization method which has been extensively used in \cite{Freefem,DE2012} allows us to extend to \hbox{$\theta<1$} results which were known only for $\theta=1$.
\item[$\bullet$] Even if they cannot be optimal as shown in Theorem~\ref{Thm:Main2}, the estimates of Theorem~\ref{Thm:Main1} are rather accurate from the numerical point of view, as will be illustrated in Section~\ref{Sec:Numerics}.
\end{itemize}

\medskip This paper is organized as follows. Section~\ref{Sec:Parametrization} is devoted to the reparametri\-zation and the proof of symmetry when $\Lambda\le\Lambda_1(\theta,p,d)$ in the subcritical case $\vartheta(p,d)<\theta<1$. To the price of some additional technicalities, the range $\Lambda\le\Lambda_2(\theta,p,d)$ and $\vartheta(p,\min\{3,d\})\le\theta<1$ is covered in Section~\ref{Sec:Lambda2}. The proofs of Theorems~\ref{Thm:Main1} and~\ref{Thm:Main2} are established in Section~\ref{Sec:Proofs}. The last section is devoted to an explicit approximation of $\Lambda_0$ and $\Lambda_2$, and some numerical results which illustrate Theorems~\ref{Thm:Main1} and~\ref{Thm:Main2}. The reader interested in the strategy of the proofs as well as the origin of the expressions of $\Lambda_1(\theta,p,d)$ and $\Lambda_2(\theta,p,d)$ is invited to read first Section~\ref{Sec:Parametrization} and the proof of Lemma~\ref{Lem:Second} in Section~\ref{Sec:Lambda2}.

\section{Reparametrization and a first symmetry result}\label{Sec:Parametrization}

We begin by a reparametrization of the branches of the solutions which allows us to reduce the case corresponding to $\theta<1$ and $\Lambda$ to the case corresponding to $\theta=1$ and some related $\mu$, as in Theorem~\ref{Thm:DEL2011}. Consider an optimal function~$u$ for~\eqref{Ineq:Gen_interp_Cylinder}, which therefore satisfies
\[
\K{CKN}(\theta,p,\Lambda)=\mathcal Q_\Lambda^\theta[u]=(t+\Lambda)^\theta\,\frac{\nrmC u2^2}{\nrmC up^2}\quad\mbox{with}\quad t:=\frac{\nrmC{\nabla u}2^2}{\nrmC u2^2}\ .
\]
According to \cite[Theorem1]{springerlink:10.1007/s00526-011-0394-y}, such a function $u$ exists for any $\theta>\vartheta(p,d)$. As a critical point of $\mathcal Q_\Lambda^\theta$, $u$ solves \eqref{EL} with
\[\label{muu}
\theta\,\mu=(1-\theta)\,t+\Lambda
\]
if it has been normalized by the condition
\[\label{normalization}
\nrmC{\nabla u}2^2+\Lambda\,\nrmC u2^2=\theta\,\nrmC up^p\ .
\]
Because of the zero-homogeneity of $\mathcal Q_\Lambda^\theta$, such a condition can be imposed without restriction and is equivalent to
\be{ElimL2}
\nrmC u2^2=\frac \theta{t+\Lambda}\,\nrmC up^p\ .
\ee
\begin{proposition}\label{Prop:propcond1} Let us assume that $u$ is a solution of \eqref{EL}, satisfying $\mathcal Q_\Lambda^\theta[u]=\K{CKN}(\theta,p,\Lambda)$ and \eqref{ElimL2}, with $\theta\,\mu=(1-\theta)\,t+\Lambda$. Then we have
\be{Cdt2}
\mathcal Q_\mu^1[u]\le\K{CKN}^*(1,p,\mu)\ .
\ee
\end{proposition}
\begin{proof} From \eqref{K-Kstar} we know that
\[
(t+\Lambda)^\theta\,\frac{\nrmC u2^2}{\nrmC up^2}\le\K{CKN}^*(\theta,p,\Lambda)\ .
\]
Using \eqref{ElimL2}, we rewrite this estimate as
\[
\theta\,(t+\Lambda)^{\theta-1}\,\|u\|_p^{p-2}\le\K{CKN}^*(\theta,p,\Lambda)\ .
\]
Using \eqref{ElimL2} again and the expression of $\mu$, we obtain 
\[
Q^1_{\mu}[u]=\frac{\theta\,(t+\mu)}{t+\Lambda}\,\|u\|_p^{p-2}=\|u\|_p^{p-2}\le f(t,\theta,\Lambda,p)\,\K{CKN}^*(1,p,\mu)
\]
with
\[
f(t,\theta,\Lambda,p):=\frac1{\theta\,(t+\Lambda)^{\theta-1}}\,\frac{\K{CKN}^*(\theta,p,\Lambda)}{\K{CKN}^*(1,p,\mu)}\ .
\]
Using the expression of $\mu$ and~\eqref{EqnK}, we find that
\[
f(t,\theta,\Lambda,p)=\tfrac{(p+2)^\frac{p+2}{2\,p}}{(2\,p)^{1-\theta}}\,\big(\tfrac{\Lambda\,\theta}{2+(2\,\theta-1)\,p}\big)^{\theta-\frac{p-2}{2\,p}}\,(t+\Lambda)^{1-\theta}\,\big((1-\theta)\,t+\Lambda\big)^{-\frac{p+2}{2\,p}}
\]
achieves its maximum at $t_0:=\Lambda\,\big(\frac{2\,p\,\theta}{p-2}-1\big)^{-1}>0$. Hence $f(t)\le f(t_0)=1$, which concludes the proof.
\end{proof}

Using the notations~\eqref{m5} and~\eqref{qstar}, we obtain our first symmetry result, which goes as follows.
\begin{lemma}\label{Lem:First} Suppose that either $d =2$ and $p\in(2,6)$ or else $d\ge 3$, $p\in(2,2^*)$. If $\theta\in(\vartheta(p,d),1)$ and
\[
\Lambda\le\min\left\{\Lambda_\star(1,q,d),\frac{\theta\,\Lambda_\star(1,p,d)}{(1-\theta)\,\Pi^*(\theta,p,q)+\theta}\right\}
\]
for some $q\in\big(q^*(\theta,p),6)$ when $d=2$, or for some $q\in\big(q^*(\theta,p),2^*\big)$ when $d \ge 3$, then any optimal function for \eqref{Ineq:Gen_interp_Cylinder} is symmetric.\end{lemma}
\begin{proof} Let $u$ be a solution as in Proposition~\ref{Prop:propcond1}. From~\eqref{K-Kstar}, we know that
\[
\K{CKN}^*(\theta,p,\Lambda)\ge(t+\Lambda)^\theta\,\frac{\nrmC u2^2}{\nrmC up^2}\ .
\]
For $p < q < \min\{6, 2^* \}$ we have by
H\"older's inequality, $\nrmC up\le\nrmC u2^\delta\,\nrmC uq^{1-\delta}$ provided $\delta=\frac2p\,\frac{q-p}{q-2}$, and thus $1-\delta=\frac qp\,\frac{p-2}{q-2}$. Hence
\[
\K{CKN}^*(\theta,p,\Lambda)\ge(t+\Lambda)^\theta\,\(\frac{\nrmC u2^2}{\nrmC uq^2}\)^{1-\delta}\ .
\]
Now, for any $\lambda\in(0,\Lambda_\star(1,q,d)]$, we know from Theorem~\ref{Cor:DEL2011} that
\[\label{condlambda}
\nrmC uq^2\le\frac{\nrmC{\nabla u}2^2+\lambda\,\nrmC u2^2}{\K{CKN}^*(1,q,\lambda)}\ ,
\]
which shows that
\[
\K{CKN}^*(\theta,p,\Lambda)\ge(t+\Lambda)^\theta\,\(\frac{\K{CKN}^*(1,q,\lambda)}{t+\lambda}\)^{1-\delta}\ .
\]
Summarizing, we have found that
\be{m000}
\frac{(t+\Lambda)^\theta}{(t+\lambda)^{1-\delta}}\le\frac{\K{CKN}^*(\theta,p,\Lambda)}{(\K{CKN}^*(1,q,\lambda))^{1-\delta}}\quad\mbox{if}\quad\lambda\le\Lambda_\star(1,q,d)\ .
\ee

Next we can make the ansatz $\lambda=\Lambda$. Provided $\Lambda\le\Lambda_\star(1,q,d)$, we get that
\[
(t+\Lambda)^{\theta+\delta-1}\le\frac{\K{CKN}^*(\theta,p,\Lambda)}{(\K{CKN}^*(1,q,\Lambda))^{1-\delta}}=\big(\Pi^*(\theta,p,q)\,\Lambda\big)^{\theta+\delta-1}\ ,
\]
so that $t\le(\Pi^*(\theta,p,q)-1)\,\Lambda$. According to Theorem~\ref{Thm:DEL2011}, $u$ is symmetric if 
\be{Cdt1}
\frac1\theta\,\big((1-\theta)\,t+\Lambda\big)=\mu\le\Lambda_\star(1,p,d)\ ,
\ee
because \eqref{Cdt2} holds by Proposition~\ref{Prop:propcond1}. This completes the proof. \end{proof}

In the next section we shall consider an alternative ansatz for which $\lambda\neq\Lambda$.

\section{Another symmetry result}\label{Sec:Lambda2}

In this section we establish an estimate similar to the one of Lemma²~\ref{Lem:First} but based on a different ansatz, which moreover covers the critical case $\theta=\vartheta(p,d)$. We recall that $\beta=\beta(\theta,p)=1-\frac{p-2}{2\,p\,\theta}$ has been defined in \eqref{beta}. The proof is slightly more technical than the one of Lemma²~\ref{Lem:First}. We start with an auxiliary result.
\begin{lemma}\label{Lem:xstar} For any $\mathsf N>1$, $p<6$ and $\theta\in(\vartheta(p,3),1)$, if $\beta=\beta(\theta,p)$ is given by~\eqref{beta}, the equation
\[
\theta\,(6-p)\,\big(x^\beta-\mathsf N\big)\,x-\big(2\,p\,\theta-3\,(p-2)\big)\(\theta\,\big(x^\beta-\mathsf N\big)+(1-\theta)\,(x-1)\,\mathsf N \) =0 \ ,
\]
has a unique root in the interval $(\mathsf N^{1/\beta},\infty)$.
\end{lemma}
\noindent When $\mathsf N= \mathsf N(\theta, p)>1$ is given by~\eqref{Ntp}, we denote this root by $\mathsf x^* = \mathsf x^*(\theta, p)$.
\begin{proof} Consider the function
\[
f(x):=\textstyle\theta\,(6-p)\big(x^\beta-\mathsf N\big)\,x-\big(2\,p\,\theta-3\,(p-2)\big)\!\left[\theta\big(x^\beta\!-\!\mathsf N\big)+(1\!-\!\theta)(x\!-\!1)\mathsf N \right] ,
\]
and notice first that $f(\mathsf N^{1/\beta})<0$ because $\theta>\vartheta(p,3)$ and $\mathsf N^{1/\beta}>1$. Next we observe that $\alpha:=2\,p\,\theta-3\,(p-2)=2\,p\,\big(\theta-\vartheta(p,3)\big)=6-p-2\,p\,(1-\theta)$ and compute
\[
f'(x) = (6-p)\,\theta\,\left[ (1+\beta)\,x^{\beta} - \mathsf N \right] - 2\,p\,\big(\theta-\vartheta(p,3)\big)\!\left[ \beta\,\theta\, x^{\beta-1} + (1-\theta)\,\mathsf N \right]
\]
and
\[
f^{''}(x) =\beta\,\theta\,x^{\beta-2}\,\left[(6-p)\,(1+\beta)\,x-(\beta-1)\big(6-p-2\,p\,(1-\theta)\big)\right]>0
\]
for any $x>1$. Using the fact that $\mathsf N > 1$, we find that 
\[
f'(\mathsf N^{1/\beta}) \geq 2\,(p-2)\,(1-\theta)\,\mathsf N >0 \ .
\]
It follows that the function $f(x)$ is increasing and convex for $x > \mathsf N^{1/\beta}$. Since $f(\mathsf N^{1/\beta})<0$ we conclude that $f(x)$ has a unique root for $x > \mathsf N^{1/\beta}$.
\end{proof} 

When $\mathsf N= \mathsf N(\theta, p)$ we only need to check that $\mathsf N(\theta, p)>1$. This is shown in Lemma ~\ref{Cor:10}. Before, we need a preliminary estimate. Consider the \emph{Digamma function} $\psi(z)=\frac{\Gamma'(z)}{\Gamma(z)}$.
\begin{lemma}\la{psi} For all $z>0$, we have
\[
\frac1{2\,z}<\psi\(z+\tfrac12\)-\psi(z)<\ln\(1+\tfrac1{2\,z}\)+\frac1{z}-\frac2{2\,z+1}\ .
\]
\end{lemma}
\begin{proof}We use the following representation formula (\emph{cf.} \cite[\textsection~6.3.21, p.~Ê259]{AS}):
\[
\psi(z)=\int_0^\infty\(\frac{e^{-t}}t-\frac{e^{-z\,t}}{1-e^{-t}}\)\,dt
\]
and elementary manipulations to get the lower bound
\[
\psi\(z+\tfrac12\)-\psi(z)=\int_0^\infty \frac{e^{-z\,t}}{1+e^{-t/2}}\;dt>\frac12\int_0^\infty e^{-z\,t}\;dt=\frac1{2\,z}\ .
\]
As for the upper bound, we have the equivalences
\bean
&&\ln\(1+\tfrac1{2\,z}\)+\frac1{z}-\frac2{2\,z+1}-\int_0^\infty \frac{e^{-z\,t}}{1+e^{-t/2}}\;dt>0\\
&&\Longleftrightarrow
\ln\(1+\tfrac1{2\,z}\)+\int_0^\infty e^{-z\,t}\;dt-\frac2{2\,z+1}-
\int_0^\infty \frac{e^{-z\,t}}{1+e^{-t/2}}\;dt>0\\
&&\Longleftrightarrow
\ln\(1+\tfrac1{2\,z}\)+\int_0^\infty \frac{e^{-t/2}\,e^{-z\,t}}{1+e^{-t/2}}\;dt-\frac2{2\,z+1}>0\ .
\eean
The result follows from
\[
\int_0^\infty \frac{e^{-(z+\frac12)\,t}}{1+e^{-t/2}}\;dt>\frac12\int_0^\infty e^{-(z+\frac12)\,t}\;dt=\frac1{2\,z+1}
\]
and, by monotonicity of the function $z\mapsto\ln\(1+\tfrac1{2\,z}\)-\frac1{2\,z+1}$,
\[\label{le}
\ln\(1+\tfrac1{2\,z}\)>\frac1{2\,z+1}\ .
\]
\end{proof} 

\begin{lemma}\label{Cor:10} Assume that $2<p<6$ and $\vartheta(p,2)<\theta\le1$. Then the function $\theta\mapsto\mathsf N(\theta,p)$ is decreasing and $\mathsf N(1,p)=1$.\end{lemma}
\begin{proof} $\mathsf N(1,p)=1$ is a consequence of the definition of $\mathsf N$. Using the precise value pf $ \K{CKN}^*(\theta,p,\Lambda) $, we obtain the following explicit expression of the function $\mathsf N(\theta,p)$, namely
\[
\textstyle\(\frac2{2-p\,(1-\theta)}\)^{\!\frac{p-2}{2\,p\,\theta}}\(\frac{p+2}4\)^{\!\frac{6-p}{2\,p\,\theta}} \(\frac{2\,(2-p\,(1-\theta))}{(2\,\theta-1)\,p+2}\)^{\!\frac{2\,p\,\theta-3\,(p-2)}{2\,p\,\theta}}\left[\frac{\Gamma\(\frac2{p-2}\)\,\Gamma\(\frac{2-p\,(1-\theta)}{p-2}+\frac12\)}{\Gamma\(\frac2{p-2}+\frac12\)\,\Gamma\(\frac{2-p\,(1-\theta)}{p-2}\)}\right]^{\frac{p-2}{p\,\theta}}\kern-6pt.
\]
Let us define $G:=\mathsf N^\theta$ and compute
\bean
\frac1G\,\frac{\partial G}{\partial\theta}=\frac{p\,\theta-2\,(p-2)}{2-p\,(1-\theta)}-\frac{2\,p\,\theta-3\,(p-2)}{(2\,\theta-1)\,p+2}+\ln\(\frac{2\,(2-p\,(1-\theta))}{(2\,\theta-1)\,p+2}\)\\
+\frac{\Gamma'\(\frac{2-p\,(1-\theta)}{p-2}+\frac12\)}{\Gamma\(\frac{2-p\,(1-\theta)}{p-2}+\frac12\)}-\frac{\Gamma'\(\frac{2-p\,(1-\theta)}{p-2}\)}{\Gamma\(\frac{2-p\,(1-\theta)}{p-2}\)}\ .
\eean
By Lemma~\ref{psi} we get that
\bean
\frac1G\,\frac{\partial G}{\partial\theta}<\frac{p\,\theta-2\,(p-2)}{2-p\,(1-\theta)}-\frac{2\,p\,\theta-3\,(p-2)}{(2\,\theta-1)\,p+2}+\ln\(\frac{2\,(2-p\,(1-\theta))}{(2\,\theta-1)\,p+2}\)\\
+\ln\(\frac{(2\,\theta-1)\,p+2}{2\,(2-p\,(1-\theta))}\)+\frac{p-2}{2-p\,(1-\theta)}-\frac{2\,(p-2)}{(2\,\theta-1)\,p+2}=0\ .
\eean
Since
\[
\frac1G\,\frac{\partial G}{\partial\theta}=\ln \mathsf N+\frac\theta{\mathsf N} \frac{\partial \mathsf N}{\partial\theta}<0\ ,
\]
for $\theta\in(\vartheta(p,2), 1]$ and $\mathsf N(1,p)=1$, it follows that $\frac\partial{\partial\theta}\mathsf N<0$.
\end{proof}

After these preliminaries, we can now state the main result of this section.
\begin{lemma}\label{Lem:Second} Assume that
\begin{eqnarray*}
&&2<p<6\quad\mbox{and}\quad\vartheta(p,3)<\theta<1\quad\mbox{if}\quad d=2\;\mbox{ or }\;3\ ,\\
&&2<p<2^*\quad\mbox{and}\quad\vartheta(p,d)\le\theta<1\quad\mbox{if}\quad d\ge4\ .
\end{eqnarray*}
Then any optimal function for \eqref{Ineq:Gen_interp_Cylinder} is symmetric if $\Lambda<\Lambda_2(\theta,p,d)$.
Moreover, we have $\lim_{\theta\to1}\Lambda_2(\theta,p,d)=\Lambda_\star(1,p,d)$.\end{lemma}
\begin{proof} As in the proof of Lemma~\ref{Lem:First}, the starting point of our estimate is inequality \eqref{m000}, which becomes
\[\label{k1}
\frac{t+\Lambda}{t+\lambda}\le\mathsf N(\theta,p)\(\frac\Lambda\lambda\)^\beta\quad\mbox{if}\quad\lambda<\Lambda_\star(1,q,d)\ ,
\]
under the restriction that we choose $q=q^*(\theta,p)$ given by~\eqref{qstar}, that is $1-\delta=\theta$ with $\delta$ as in \eqref{m5}. Remarkably, we observe that, for this specific value of $q$, we have
\[
\theta-\frac{p-2}{2\,p}=\(1-\frac{q-2}{2\,q}\)(1-\delta)
\]
and, as a consequence, 
\[
\frac{t+\Lambda}{t+\lambda}\le\mathsf N\(\frac\Lambda\lambda\)^\beta
\]
where $\beta:=1-\frac{p-2}{2\,p\,\theta}$ and $\mathsf N=\mathsf N(\theta,p)$. Hence we get that
\[
t\le\frac{\mathsf N\,\Lambda^\beta\,\lambda-\Lambda\,\lambda^\beta}{\lambda^\beta-\mathsf N\,\Lambda^\beta}=:\bar t\ .
\]
As in the proof of Lemma~\ref{Lem:First}, we can apply Theorem~\ref{Thm:DEL2011} if
\begin{itemize}
\item[$\bullet$] Condition~\eqref{Cdt1} holds and a sufficient condition is therefore given by the condition
\[
(1-\theta)\,\bar t+\Lambda\le\theta\,\Lambda_\star(1,p,d)\ ,
\]
that is,
\[
\big(\theta\,\Lambda_\star(1,p,d)-\Lambda\big)\(\lambda^\beta-\mathsf N\,\Lambda^\beta\)\ge(1-\theta)\(\mathsf N\,\Lambda^\beta\,\lambda-\Lambda\,\lambda^\beta\).
\]
\item[$\bullet$] Condition $\lambda <\Lambda_\star(1,q,d)$, which is required to get~\eqref{m000}, holds, \emph{i.e.},
\[
\lambda<\Lambda_\star(1,q,d)=\frac14\,(d-1)\,\frac{6-q}{q-2}=\frac14\,(d-1)\,\frac{2\,p\,\theta-3\,(p-2)}{p-2}\ .
\]
\end{itemize}
For a suitable $x=\lambda/\Lambda> \mathsf N^{1/\beta}$, to be chosen, these two conditions amount to
\begin{eqnarray*}
&&\Lambda\le\phi(x):=\frac{\theta\,\Lambda_\star(1,p,d)\,\big(x^\beta-\mathsf N\big)}{\theta\,(x^\beta-\mathsf N)+(1-\theta)\,(x-1)\,\mathsf N}\ ,\\
&&\Lambda<\chi(x):=\frac14\,(d-1)\,\frac{2\,p\,\theta-3\,(p-2)}{p-2}\,\frac 1x\ .
\end{eqnarray*}
After replacing $\Lambda_\star(1,p,d)$ by its value according to~\eqref{lambdastar}, we get that $\phi(x)-\chi(x)$ has the sign of $f(x)$ as defined in the proof of Lemma~\ref{Lem:xstar}. By Corollary~\ref{Cor:10}, we know that $\mathsf N\ge1$ and conclude henceforth that any minimizer is \emph{symmetric} if $\Lambda<\chi(\mathsf x^*(\theta,p))=\Lambda_2(\theta,p,d)$.

In the limiting regime corresponding to as $\theta\to1_-$, we observe that $\phi(x)=\Lambda_\star(1,p,d)$ and $\chi(x)=\Lambda_\star(1,p,d)/x$, so that $\lim_{\theta\to1}\Lambda_2(\theta,p,d)=\chi(1)=\Lambda_\star(1,p,d)$.
\end{proof}

\section{Proof of the main results}\label{Sec:Proofs} 

\begin{proof}[Theorem~\ref{Thm:Main1}] It is a straightforward consequence of Lemma~\ref{Lem:First} and Lem\-ma~\ref{Lem:Second}. Notice that $\lim_{\theta\to1}\Lambda_1(\theta,p,d)=\Lambda_\star(1,p,d)$ because
\[
\lim_{\theta\to1}\frac{\theta\,\Lambda_\star(1,p,d)}{(1-\theta)\,\Pi^*(\theta,p,q)+\theta}=\Lambda_\star(1,p,d)\ .\]\end{proof}

\begin{proof}[Theorem~\ref{Thm:Main2}] The function $q\mapsto\Lambda_1(1,q,d)$ is monotone decreasing and
\[
q^*(\theta,p)-p=\frac{p\,(p-2)\,(1-\theta)}{2-p\,(1-\theta)}\ge0
\]
so that, for $i=1$, $2$,
\[
\Lambda_i(\theta,p,d)\le\Lambda_\star(1,q^*(\theta,p),d)\le\Lambda_\star(1,p,d)<\Lambda_{\rm FS}(\theta,p,d)\ .
\]

By definition of $\Lambda_{\rm s}(\theta,p,d)$, we know that $\Lambda_\star(\theta,p,d)\le\Lambda_{\rm s}(\theta,p,d)$. By Theorem~\ref{Thm:Main1}, if $\Lambda=\Lambda_\star(\theta,p,d)$ any minimizer for $\K{CKN}(\theta,p,\Lambda)$ is symmetric. On the other hand, by continuity, we know that
\[
\K{CKN}\big(\theta,p,\Lambda_{\rm s}(\theta,p,d)\big)=\K{CKN}^*\big(\theta,p,\Lambda_{\rm s}(\theta,p,d)\big)\ .
\]

Let us assume that $\Lambda_{\rm s}(\theta,p,d)<\Lambda_{\rm FS}(\theta,p,d)$ and consider a sequence $(\lambda_n)_{n\in\N}$ converging to $\Lambda_{\rm s}(\theta,p,d)$ with $\lambda_n>\Lambda_{\rm s}(\theta,p,d)$. If $u_n$ is a non-symmetric minimizer of $\K{CKN}(\theta,p,\lambda_n)$, we can pass to the limit: up to the extraction of a subsequence, $(u_n)_{n\in\N}$ converges in $\H^1(\mathcal C)$ towards a minimizer $u$ for $\K{CKN}(\theta,p,\Lambda_{\rm s}(\theta,p,d))$. The function $u$ cannot only depend on $s$, because any symmetric minimizer for $\K{CKN}^*(\theta,p,\Lambda)$ is a strict local minimum in $\H^1(\mathcal C)$ due to the fact that $\Lambda_{\rm s}(\theta,p,d)<\Lambda_{\rm FS}(\theta,p,d)$. Hence, for $\Lambda=\Lambda_{\rm s}(\theta,p,d)$ there are two distinct minimizers for $\K{CKN}(\theta,p,\Lambda)$: one is symmetric and the other one \emph{is not} symmetric. This proves that $\Lambda_\star(\theta,p,d)<\Lambda_{\rm s}(\theta,p,d)$ if $\theta>\vartheta(p,\min\{3,d\})$.

In the other cases, that is, if either $d=3$ and $\theta=\vartheta(p,3)$, or $d=2$ and $\theta>0$, the same method applies if we replace $\Lambda_\star(\theta,p,d)$ by $\Lambda_2(\theta,p,d)$.\end{proof}

\begin{proof}[Theorem~\ref{Cor:WLH}] Let us consider $f(x)$ as in the proof of Lemma~\ref{Lem:xstar} and assume that $\theta=\gamma\,(p-2)$. As $p\to2_+$, $f(x)/(p-2)$ converges towards
\[
f_0(x):=4\,\gamma\,x^{\beta_0+1}-(8\,\gamma-3)\,\mathsf N_0 \, x+(4\,\gamma-3)\,\mathsf N_0 \quad\mbox{with}\quad\beta_0=1-\tfrac1{4\,\gamma}\ .
\]
We easily check that the function $f_0(x)$ is convex for $x>0$, $f_0(\mathsf N_0^{1/\beta_0})<0$ and 
$f_0'(\mathsf N_0^{1/\beta_0})= 2\,\mathsf N_0 >0$.
We conclude that $f_0(x)$ has a unique root for $x > \mathsf N_0^{1/\beta_0}$. 
We denote this unique root by $\mathsf x_0^* =\mathsf x_0^*(\gamma)$.
It follows that $\mathsf x^*(\gamma\,(p-2),p)$ converges to $\mathsf x_0^*(\gamma)$ as $p\to2_+$.
Symmetry then is established by passing to the limit for any $\Lambda\in\big(0,\Lambda_0(\gamma,d)\big)$ with $\Lambda_0(\gamma,d)$ given by~\eqref{Lambda0}.
\end{proof}

\section{An approximation and some numerical results}\label{Sec:Numerics}

The functions $\mathsf x^*(\theta,p)$ and $\mathsf x_0^*(\gamma)$ which enter in the results of Theorem~\ref{Thm:Main1} and Theorem~\ref{Cor:WLH} are not explicit but easy to estimate, which in turn gives explicit estimates of $\Lambda_2(\theta,p,d)$ and $\Lambda_0(\gamma,d)$. Let
\begin{eqnarray*}
&&\alpha=2\,p\,\big(\theta-\vartheta(p,3)\big)=2\,p\,\theta-3\,(p-2)\ ,\\
&&\beta=\beta(\theta,p)=1-\frac{p-2}{2\,p\,\theta} \ ,
\end{eqnarray*}
and
\[
\Lambda_{2,\rm approx} (\theta,p,d) :=\frac{(d-1)\,\xa}{4\,(p-2)}\,\frac{\beta\,\theta\,(6-p)-\alpha\,(1-\theta + \beta\,\theta\,\mathsf N^{-1/\beta})}{\beta\,\theta\,(6-p)\,\mathsf N^{1/\beta} -\alpha\,(\beta\,\theta +1-\theta)}\ .
\]
\begin{proposition}\label{Prop:Estim1} Suppose that either $d =2$ and $p\in(2,6)$ or else $d\ge 3$ and $p\in(2,2^*)$.
Then for any $\theta \in (\vartheta(p,3),1)$, we have the estimate
\[
\Lambda_2(\theta,p,d)>\Lambda_{2,\rm approx}(\theta,p,d)\ .
\]
\end{proposition}
\begin{proof} Let us consider the function $f$ defined in the proof of Lemma~\ref{Lem:xstar} and recall that $f''(x)$ is positive for any $x\ge\mathsf N^{1/\beta}>1$. Moreover we verify that
\begin{eqnarray*}
f(\mathsf N^{1/\beta}) & = & -\,(1-\theta)\,\alpha\,\mathsf N\,(\mathsf N^{1/\beta}-1) <0 \ , \\
f'(\mathsf N^{1/\beta})& = &\mathsf N\left[\beta\,\theta\,(6-p)-\alpha\,\big(1-\theta +\beta\,\theta\,\mathsf N^{-1/\beta}\big)\right] \, >0.
\end{eqnarray*}
which provides the estimate
\[
\mathsf x^*(\theta,p)<\mathsf N^{1/\beta}-\frac{f(\mathsf N^{1/\beta})}{ f'(\mathsf N^{1/\beta})}
= \frac{\beta\,\theta\,(6-p)\,\mathsf N^{1/\beta}-\alpha\,(\beta\,\theta+1-\theta)}{\beta\,\theta\,(6-p)-\alpha\,\big(1-\theta +\beta\,\theta\,\mathsf N^{-1/\beta}\big)} \ ,
\]
and the result follows.
\end{proof}

Next we give an estimate of $\Lambda_0(\gamma,d)$ in Theorem~\ref{Cor:WLH}. Let
\[
\Lambda_{0,\rm approx}(\gamma,d):= \frac{(d-1)\,(\gamma-\frac34)}{2\(\gamma-\frac14\)\mathsf N_0^{\frac{4 \xg}{4 \xg-1}} - 2 \(\gamma-\frac34\) } \ ,
\]
with $\mathsf N_0(\gamma)$ as defined by \eqref{N0}.
\begin{proposition}\label{Prop:Estim2} Assume that $d \geq 2$ and $\xg>3/4$. Then
\[
\Lambda_0(\gamma,d)> \Lambda_{0,\rm approx}(\gamma,d) \ .
\]
\end{proposition}
\begin{proof} Recall that $\beta_0=1-\frac1{4\,\gamma}$. Let us consider the function $f_0$ defined in the proof of Theorem~\ref{Cor:WLH}. We note that $f_0''(x)$ is positive for $x>0$. Moreover we verify that $f_0'(\mathsf N_0^{1/\beta_0})=2\,\mathsf N_0>0$ and $f_0(\mathsf N_0^{1/\beta_0})=-(4\,\gamma-3)\,\mathsf N_0\,(\mathsf N_0^{1/\beta_0}-1) <0$, which provides the estimates
\[\textstyle
\mathsf x_0^*(\gamma)<\mathsf N_0^{1/\beta_0}-\frac{f(\mathsf N_0^{1/\beta_0})}{ f'(\mathsf N_0^{1/\beta})}= 2\(\xg-\frac14\) \mathsf N_0^{\frac{4 \xg}{4 \xg-1}} - 2\(\xg-\frac34\) \ ,
\]
and the result follows.
\end{proof}

To conclude this paper, let us illustrate Theorems~\ref{Thm:Main1} and~\ref{Thm:Main2} with some numerical results. First we address the case of subcritical $\theta\in(\vartheta(p,d),1)$ and compare $\Lambda_\star$ with $\Lambda_{\rm FS}$: Fig.~\ref{F1} corresponds to the particular case $d=5$ and $\theta=0.5$.

The expression of $\Lambda_\star(\theta,p,d)$ is not explicit but easy to compute numerically. We recall that $\Lambda_\star$ is the maximum of $\Lambda_1$ and $\Lambda_2$, both of them being non-explicit. In practice, for low values of the dimension $d$, the relative difference of $\Lambda_1$ and $\Lambda_2$ is in the range of a fraction of a percent to a few percents, depending on $\theta$ and on the exponent $p$. Moreover, we numerically observe that $\Lambda_1\le\Lambda_2$, at least for the values of the parameters considered in Fig.~\ref{F1}. The estimate $\Lambda_{2,\rm approx}(\theta,p,d)$ of Proposition~\ref{Prop:Estim1} is remarkably good.

\begin{figure}[ht]
\begin{center}
\includegraphics[width=9cm]{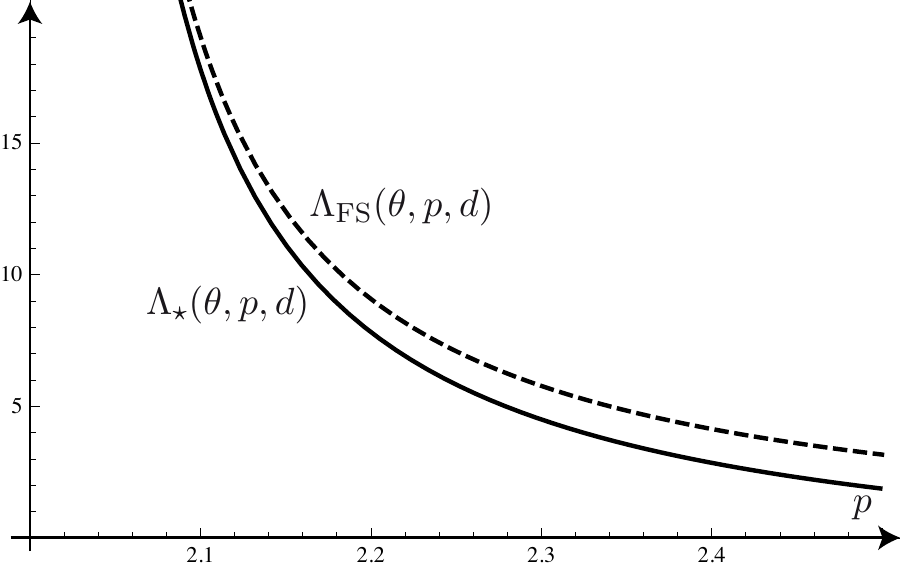}
\end{center}
\caption{\scriptsize\label{F1} Curves $p\mapsto\Lambda_\star(\theta,p,d)$ and $\Lambda\mapsto\Lambda_{\rm FS}(\theta,p,d)$ with $\theta=0.5$ and $d=5$. Symmetry holds for $\Lambda\le\Lambda_\star(\theta,p,d)$, while symmetry is broken for $\Lambda\ge\Lambda_{\rm FS}(\theta,p,d)$. The relative difference of $\Lambda_1$ and $\Lambda_2$, \emph{i.e.}, $\Lambda_2(\theta,p,d)/\Lambda_1(\theta,p,d)-1$, is below 4\%. The estimate of Proposition~\ref{Prop:Estim1} is such that $1-\Lambda_{2,\rm approx}(\theta,p,d)/\Lambda_2(\theta,p,d)$ is of the order of $5\times10^{-3}$.}
\end{figure}

\medskip In Fig.~\ref{F3}, we consider the critical case $\theta=\vartheta(p,d)$. The plot corresponds to $d=5$ and all $p$ in the interval $(2,10/3)$. The exponent $\vartheta(p,d)$ is the one which enters in the Gagliardo-Nirenberg inequality
\[
\nrmrd up^2\le\C{GN}(p,d)\,\nrmrd{\nabla u}2^{2\,\vartheta(p,d)}\,\nrmrd u2^{2\,(1-\vartheta(p,d))}\quad\forall\,u\in\H^1(\R^d)
\]
on the Euclidean space $\R^d$, \emph{without weights}. Here $\C{GN}(p,d)$ denotes the optimal constant and $p\in(2,\infty)$ if $d=1$ or $2$, $p\in(2,2^*]$ if $d\ge3$. The optimizers are radially symmetry but not known explicitly. 

It has been shown in \cite[Theorem~1.4]{1005} that optimal functions for~\eqref{Ineq:GenInterp} exist if $\C{GN}(p,d)<\C{CKN}(\theta,p,a)$. On the other hand, optimal functions cannot be symmetric $\C{GN}(p,d)>\C{CKN}^*(\theta,p,a)$: see \cite[Section~5]{springerlink:10.1007/s00526-011-0394-y} for further details and consequences. This symmetry breaking condition determines a curve $p\mapsto\Lambda_{\rm GN}(p,d)$ which has been computed numerically in \cite{1008,Oslo}: there are values of~$p$ and $d$ for which the condition $\Lambda>\Lambda_{\rm GN}(p,d)$, which guarantees symmetry breaking (but not existence), is weaker than the condition $\Lambda>\Lambda_{\rm FS}(\theta,p,d)$, that is $\Lambda_{\rm GN}(p,d)<\Lambda_{\rm FS}(\theta,p,d)$. See Fig.~\ref{F3}. A rather complete scenario of explanations, based on numerical computations and some formal expansions, has been established in \cite{Freefem,DE2012}. As it had to be expected, we numerically observe that $\Lambda_\star(\theta,p,d)\le\min\{\Lambda_{\rm FS}(\theta,p,d),\Lambda_{\rm GN}(p,d)\}$ when $\theta=\vartheta(p,d)$, for any $p\in(2,2^*)$.
\begin{figure}[hb]
\begin{center}
\includegraphics[width=9cm]{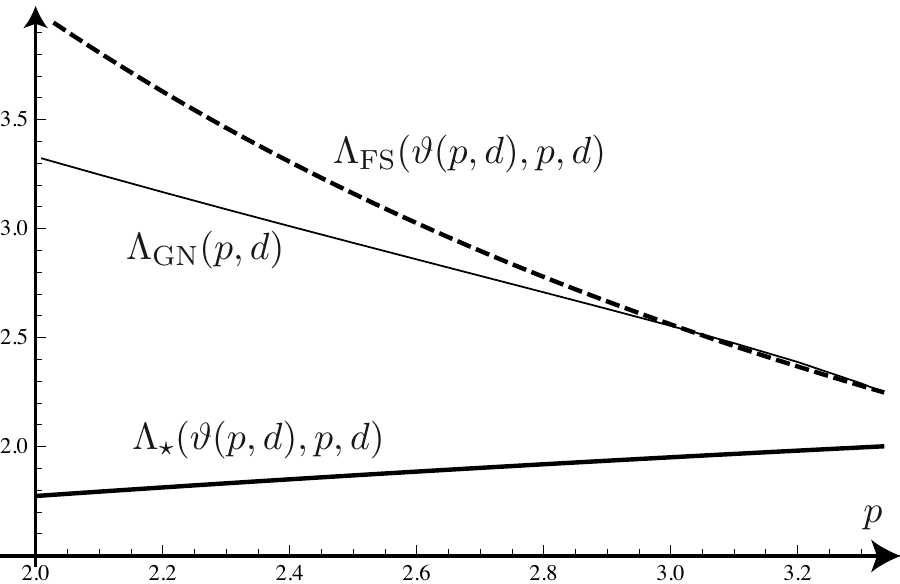}
\end{center}
\caption{\scriptsize\label{F3} With $\theta=\vartheta(p,d)$, the curve $p\mapsto\Lambda_\star(\theta,p,d)$ is always below the curves $p\mapsto\Lambda_{\rm FS}(\theta,p,d)$ and $p\mapsto\Lambda_{\rm GN}(p,d)$ for any $p\in(2,2^*)$, although $\Lambda_{\rm FS}$ and $\Lambda_{\rm GN}$ are not ordered. The plot corresponds to $d=5$ and we may notice that $\Lambda_{\rm GN}(p,d)<\Lambda_{\rm FS}$ if $p$ is small enough.}
\end{figure}


\bigskip\begin{spacing}{0.75}\noindent{\small{\bf Acknowlegments.} J.D.~thanks S.F.~and A.T.~for welcoming him in Heraklion. J.D.~and M.J.E.~have been supported by the ANR project NoNAP. J.D.~has also been supported by the ANR projects STAB and Kibord.}\\[6pt]
{\sl\scriptsize\copyright~2014 by the authors. This paper may be reproduced, in its entirety, for non-commercial purposes.}\end{spacing}
\end{document}